\documentclass[a4paper,12pt]{article}

\usepackage{graphicx}
\usepackage[utf8]{inputenc}
\usepackage[english]{babel}
\usepackage{amssymb}
\usepackage{makeidx}
\usepackage{amsmath}
\usepackage{amsthm}
\usepackage{multicol}
\usepackage{appendix}

\newtheorem{teo}{Theorem}[section]
\newtheorem{lemma}{Lemma}[section]
\newtheorem{cor}{Corollary}[section]
\newtheorem{oss}{Remark}[section]

\newcommand{\ra}{\longrightarrow}

\newcommand{\dis}{\displaystyle}

\newcommand{\bs}{\backslash}

\newcommand{\ov}{\overline}

\newcommand{\0}{\emptyset}
\newcommand{\ph}{\varphi}

\newenvironment{M}[1]{
	\left( \begin{array}{#1}}
	{\end{array}\right)}
\newenvironment{Si}[1]{\left\{\begin{array}{#1}}{\end{array} \right. }

\makeindex
\title{ }
\author{Gabriele Mancini}
\usepackage{authblk}
\title{A note on compactness properties of the singular Toda system}
\author{Luca Battaglia, Gabriele Mancini}
\affil{\small{S.I.S.S.A., Via Bonomea 265, 34136 Trieste (Italy)\\e-mail: \textit{lbatta@sissa.it}, \textit{gmancini@sissa.it}}}
\date{}
\begin{document}
\maketitle
\begin{abstract}
In this note, we consider blow-up for solutions of the $SU(3)$ Toda system on a compact surface $\Sigma$. In particular, we give a complete proof of the compactness result stated by Jost, Lin and Wang in \cite{jlw} and we extend it to the case of singularities. This is a necessary tool to find solutions through variational methods.
\end{abstract}

\section{Introduction}

Let $(\Sigma,g)$ be a smooth, compact Riemannian surface. We consider the $SU(3)$ Toda system on $\Sigma$:
\begin{equation}
\label{toda1}-\Delta u_{i} = \sum_{j=1}^2 a_{ij}\rho_{j} \left(  \frac{V_{j} e^{u_{j}}}{\int_{\Sigma}V_{j} e^{u_{j}}dv_g}  -\frac{1}{|\Sigma|} \right)-4\pi\sum_{j=1}^l\alpha_{ij}\left(\delta_{p_j}-\frac{1}{|\Sigma|}\right)\quad\quad i=1,2
\end{equation}
with $\rho_i>0$, $0<V_i\in C^\infty(\Sigma)$, $\alpha_{ij}>-1$, $p_j\in\Sigma$ given and
$$
A=(a_{ij})= \begin{pmatrix}
2 & -1\\
-1 & 2
\end{pmatrix}
$$
is the $SU(3)$ Cartan matrix.

The Toda system is widely studied in both geometry (description of holomorphic curves in $\mathbb C\mathbb P^N$, see e.g. \cite{bw,cal,cw}) and mathematical physics (non-abelian Chern-Simons vortices theory, see \cite{dunne,tar08,yang}).

In the regular case, Jost, Lin and Wang \cite{jlw} proved the following important mass-quantization result for sequences of solutions of $\eqref{toda1}$.

\begin{teo}
\label{jlw}
Suppose $\alpha_{ij}=0$ for any $i,j$ and let $u_n=(u_{1,n},u_{2,n})$ be a sequence of solutions of $\eqref{toda1}$ with $\rho_i=\rho_{i,n}$. Define, for $x\in\Sigma$, $\sigma_1(x),\sigma_2(x)$ as
\begin{equation}
\label{sigmai}
\sigma_i(x):=\lim_{r\to0}\lim_{n\to+\infty}\rho_{i,n}\frac{\int_{B_r(x)}V_{i} e^{u_{i,n}}dv_g}{\int_{\Sigma}V_{i} e^{u_{i,n}}dv_g}.
\end{equation}
Then,
\begin{equation}
\label{sigma}
(\sigma_1(x),\sigma_2(x))\in\{(0,0),(0,4\pi),(4\pi,0),(4\pi,8\pi),(8\pi,4\pi),(8\pi,8\pi)\}.
\end{equation}
\end{teo}

In the same paper, the authors state that Theorem $\ref{jlw}$ immediately implies the following compactness result.

\begin{teo}
\label{comp}
Suppose $\alpha_{ij}=0$ for any $i,j$ and let $K_1,K_2$ be compact subsets of $\mathbb R^+\backslash4\pi\mathbb N$. Then, the space of solutions of $\eqref{toda1}$ with $\rho_i\in K_i$ satisfying $\int_\Sigma u_idv_g=0$ is compact in $H^1(\Sigma)$.
\end{teo}

Theorem $\ref{comp}$ is a necessary step to find solutions of $\eqref{toda1}$ by variational methods, as was done in \cite{bjmr,mn,mr13}.\\
Although Theorem $\ref{comp}$ has been widely used, it was not explicitly proved how it follows from Theorem $\ref{jlw}$. Recently, in \cite{lwy}, a proof was given in the case $\rho_1<8\pi$.\\
The purpose of this note is to give a complete proof of Theorem $\ref{comp}$, extending it to the singular case as well. Actually, the proof follows quite directly from \cite{COS}.

In the presence of singularities, that is when we allow the $\alpha_{ij}$ to be non-zero, it is convenient to write the system $\eqref{toda1}$ in an equivalent form through the following change of variables:

$$u_{i}\to u_{i}+4\pi\sum_{j=1}^l\alpha_{ij}G_{p_j}\quad\quad\quad\mbox{where }G_p\mbox{ solves }\left\{\begin{array}{l}-\Delta G_p=\delta_p-\frac{1}{|\Sigma|}\\\int_\Sigma G_pdv_g=0\end{array}\right..$$
The new $u_{i}$'s solve
\begin{equation}
\label{toda}-\Delta u_{i} = \sum_{j=1}^2 a_{ij}\rho_{j} \left(  \frac{\widetilde V_{j} e^{u_{j}}}{\int_{\Sigma}\widetilde V_{j} e^{u_{j}}dv_g}  -\frac{1}{|\Sigma|} \right)\quad\quad i=1,2.
\end{equation}
with
$$\widetilde V_i=\Pi_{j=1}^le^{-4\pi\alpha_{ij}G_{p_j}}V_i\quad\quad\Rightarrow\quad\quad\widetilde V_i\sim d(\cdot,p_j)^{2\alpha_{ij}}\quad\mbox{near }p_j.$$

In this case, we still have an analogue of Theorem $\ref{jlw}$ for the newly defined $u_i$. The finiteness of the local blow-up values has been proved in \cite{lwz}.\\
We will also show how this quantization result implies compactness of solutions outside a closed, zero-measure set of ${\mathbb R^+}^2$.

\begin{teo}
\label{compsing}
There exist two discrete subset $\Lambda_1,\Lambda_2\subset\mathbb R^+$, depending only on the $\alpha_{ij}$'s, such that for any $K_i\Subset\mathbb R^+\backslash\Lambda_i$, the space of solutions of $\eqref{toda1}$ with $\rho_i\in K_i$ satisfying $\int_\Sigma u_idv_g=0$ is compact in $H^1(\Sigma)$.
\end{teo}

As in the regular case, Theorem $\ref{compsing}$ has an important application in the variational analysis of $\eqref{toda1}$, see for instance \cite{bjmr,bat}.

\section{Proof of the main results}

Let us consider a sequence $u_n$ of solutions of $\eqref{toda1}$ with $\rho_i=\rho_{i,n}\underset{n\to+\infty}\to\ov\rho_i$ and let us define
\begin{equation}
\label{wi}
w_{i,n}:=u_{i,n}-\log\int_{\Sigma}\widetilde V_{i} e^{u_{i,n}}dv_g+\log\rho_{i,n},
\end{equation}
which solves
\begin{equation}
\label{eqw}
\left\{\begin{array}{l}-\Delta w_{i,n} = \sum_{j=1}^2 a_{ij}\left(\widetilde V_{j} e^{w_{j,n}}-\frac{\rho_{j,n}}{|\Sigma|} \right);\\
\int_\Sigma\widetilde V_ie^{w_{i,n}}dv_g=\rho_{i,n}\end{array}\right.
\end{equation}
moreover,
$$\sigma_i(x)=\lim_{r\to0}\lim_{n\to+\infty}\int_{B_r(x)}\widetilde V_{i} e^{w_{i,n}}dv_g.$$

Let us denote by $S_i$ the blow-up set of $w_{i,n}$:
$$
S_i:=\left\{ x\in \Sigma\;:\; \exists \{x_n\}\subset \Sigma,\;  w_{i,n}(x_n)\underset{n\to+\infty}\ra +\infty \right\}.
$$

For $w_{i,n}$ we have a concentration-compactness result from \cite{LN,batmal}: 
\begin{teo}
\label{ln}
Up to subsequences, one of the following alternatives holds:
\begin{itemize} 
\item (Compactness) $w_{i,n}$ is bounded in $L^\infty(\Sigma)$ for $i=1,2$.
\item (Blow-up) The blow-up set $S:=S_1\cup S_2$ is non-empty and finite and $\forall\; i\in \{1,2\}$ either $w_{i,n}$ is bounded in $L^\infty_{loc}(\Sigma\bs S)$ or $w_{i,n}\ra -\infty$ locally uniformly in $\Sigma \bs S$.\\
In addition, if $S_i \bs (S_1 \cap S_2)\neq \0$, then $w_{i,n}\ra -\infty$ locally uniformly in $\Sigma \bs S$.
\end{itemize}
Moreover, denoting by $\mu_i$ the weak limit of the sequence of measures  $\widetilde V_{i} e^{w_{i,n}}$, one has
$$
\mu_i = r_i+ \sum_{x\in S_i} \sigma_i(x) \delta_x
$$
with $r_i \in L^1(\Sigma)\cap L^\infty_{loc}(\Sigma\bs S_i)$ and $\sigma_i(x)\ge 2\pi\min\{1,1+\alpha_i(x)\}$ $\forall x \in S_i$, $i=1,2$, where $$\alpha_i(x)=\begin{Si}{ll}
0 & \mbox{ if } x\neq p_j\; j=1,\ldots,l \\
\alpha_{ij} & \mbox{ if } x=p_j.
\end{Si}
$$
\end{teo}

Here we want to show that one has $r_i\equiv 0$ for at least one $i\in\{1,2\}$.\\
It may actually occur that only one of the $r_i$'s is zero, as shown in \cite{dpr}.\\
Anyway, to prove Theorems $\ref{comp}$ and $\ref{compsing}$ we only need one between $r_1$ and $r_2$ to be identically zero.

As a first thing, we can show that the profile near blow-up points resembles a combination of Green's functions:

\begin{lemma}\label{green}
$w_{i,n}-\ov{w}_{i,n}\ra \sum_{j=1}^2 \sum_{x \in S_j} a_{ij} \sigma_j(x)G_{x}+s_i$  in $L^\infty_{loc}(\Sigma\backslash S)$ and weakly in $W^{1,q}(\Sigma)$ for any $q\in(1,2)$ with $e^{s_i}\in L^p(\Sigma)$ $\forall p\ge 1$.
\end{lemma}
\begin{proof}
If $q\in (1,2)$ 
$$
\int_{\Sigma} \nabla w_{i,n}\;\cdot \nabla \ph dv_g\le  \|\Delta w_{i,n}\|_{L^1(\Sigma)}\|\ph\|_\infty \le  C \|\ph\|_{W^{1,q'}(\Sigma)}
$$
$\forall\; \ph \in W^{1,q'}(\Sigma)$ with  $\int_{\Sigma} \ph =0 $, hence one has $\|\nabla w_{i,n}\|_{L^q(\Sigma)}\le C$. In particular $w_{i,n}-\ov{w}_{i,n}$ converges to a function $w_i\in W^{1,q}(\Sigma)$ weakly in $W^{1,q}(\Sigma)$ $\forall q\in (1,2)$ and, thanks to standard elliptic estimates, we get convergence in $L^\infty_{loc}(\Sigma\backslash S)$.\\
The limit functions $w_i$ are distributional solutions of 
$$
-\Delta w_i = \sum_{j=1}^2 a_{ij} \left(r_j+\sum_{x \in S_j} \sigma_j(x)\delta_x -\frac{\ov{\rho}_j}{|\Sigma|}\right).
$$
In particular $s_i:= w_i - \sum_{j=1}^2\sum_{x \in S_j} a_{ij} \sigma_j(x)G_{x}$ solves
$$
-\Delta s_i = \sum_{j=1}^2 a_{ij}\left( r_j +\frac{1}{|\Sigma|}\sum_{x\in S_j} \sigma_j(x)-\frac{\ov\rho_j}{|\Sigma|} \right). 
$$
Since $-\Delta s_i \in L^1(\Sigma)$ we can exploit Remark 2 in \cite{BM} to prove that $e^{s_i}\in L^p(\Sigma)$ $\forall \;p\ge1$. 
\end{proof}


The following Lemma shows the main difference between the case of vanishing and non-vanishing residual.

\begin{lemma}\label{media}$ $
\begin{itemize}
\item $r_i\equiv 0$ $\Longrightarrow$ $\ov{w}_{i,n}\ra-\infty$.
\item $r_i\not\equiv 0$ $\Longrightarrow$ $\ov{w}_{i,n}$ is bounded.
\end{itemize}
\end{lemma}
\begin{proof}
First of all, $\ov w_{i,n}$ is bounded from above due to Jensen's inequality.\\
Now, take any non-empty open set $\Omega\Subset \Sigma\bs S$. 
$$
\int_{\Omega}\widetilde V_{i} e^{w_{i,n}}dv_g = e^{\ov{w}_{i,n}} \int_{\Omega}\widetilde V_{i} e^{w_{i,n}-\ov{w}_{i,n}}dv_g
$$
and by Lemma \ref{green}
$$
 \int_{\Omega} \widetilde V_{i} e^{w_{i,n}-\ov{w}_{i,n}} dv_g\underset{n\to+\infty}\ra\int_{\Omega} \widetilde V_{i} e^{\sum_{j=1}^2 \sum_{x \in S_j} a_{ij} \sigma_j(x)G_{x}+s_i}dv_g \in(0,+\infty).
$$
On the other hand,
$$
\int_{\Omega} \widetilde V_{i} e^{w_{i,n}} dv_g\underset{n\to+\infty}\ra \mu_i(\Omega) = \int_{\Omega} r_i(x) dv_g(x).
$$
If $r_i\equiv 0$ one has $\ov{w}_{i,n}\ra -\infty$. If instead $r_i\not\equiv0$, choosing $\Omega$ such that $\int_{\Omega} r_i(x) dv_g>0$ we must have $\ov{w}_{i,n}$ necessarily bounded.
\end{proof}

\begin{oss}
From the previous two lemmas, we can write $r_i=\widehat V_ie^{s_i}$, where
$$\widehat V_i:=\widetilde V_ie^{\lim_{n\to+\infty}\ov w_{i,n}}e^{\sum_{j=1}^2 \sum_{x \in S_j} a_{ij} \sigma_j(x)G_{x}}$$
satisfies $\widehat V_i\sim d(\cdot,x)^{2\alpha_i(x)-\frac{\sum_{j=1}^2a_{ij}\sigma_j(x)}{2\pi}}$ around each $x\in S_i$, provided $r_i\not\equiv0$.
\end{oss}

The key Lemma is an extension of Chae-Ohtsuka-Suzuki \cite{COS} to the singular case. Basically, it gives necessary conditions on the $\sigma_i$'s to have non-vanishing residual.

\begin{lemma}
\label{residuo}
For both $i=1,2$ we have $s_i\in W^{2,p}(\Sigma)$ for some $p>1$. Moreover, if $\sum_{j=1}^2a_{ij}\sigma_j(x_0)\ge4\pi(1+\alpha_i(x_0))$ for some $x_0\in S_i$, then $r_i\equiv 0 $.
\end{lemma}
\begin{proof}
If both $r_1$ and $r_2$ are identically zero, then also $s_1$ and $s_2$ are both identically zero, so there is nothing to prove.\\
Suppose now $r_1\not\equiv0$ and $r_2\equiv0$. In this case,
$$\left\{\begin{array}{l}-\Delta s_1=2\left( r_1 +\frac{1}{|\Sigma|}\sum_{x_0\in S_1} \sigma_1(x_0)-\frac{\ov\rho_1}{|\Sigma|} \right) \\-\Delta s_2=-\left( r_1 +\frac{1}{|\Sigma|}\sum_{x_0\in S_1} \sigma_1(x_0)-\frac{\ov\rho_1}{|\Sigma|} \right) \end{array}\right..$$
Then, being $G_x(y)\ge-C$ for all $x,y\in\Sigma$ with $x\ne y$, we get
$$s_1(x)=\int_\Sigma G_x(y)2r_1(y)dv_g(y)\ge-2C\int_\Sigma r_1dv_g\ge-C'.$$
Therefore, from the previous remark, around each $x_0\in S_1$ we get
$$r_1(y)\ge Cd(x_0,y)^{2\alpha_1(x_0)-\frac{\sum_{j=1}^2a_{1j}\sigma_j(x_0)}{2\pi}},$$
so being $r_1\in L^1(\Sigma)$, it must be $\sum_{j=1}^2a_{1j}\sigma_j(x_0)<4\pi(1+\alpha_1(x_0))$.\\
Moreover, being $e^{qs_1}\in L^1(\Sigma)$ for any $q\ge1$, from Holder's inequality we get $r_1\in L^p(\Sigma)$ for some $p>1$; therefore, standard estimates yield $s_i\in W^{2,p}(\Sigma)$ for both $i=1,2$.\\
Consider now the case of both non-vanishing residuals, which means by Theorem $\ref{ln}$ $S_1=S_2=S$. In this case,
$$-\Delta\left(\frac{2s_1+s_2}3\right)=\left( r_1 +\frac{1}{|\Sigma|}\sum_{x_0\in S_1} \sigma_1(x_0)-\frac{\ov\rho_1}{|\Sigma|} \right)$$
hence, arguing as before, $\frac{2s_1+s_2}3\ge-C$. Therefore, using the convexity of $t\to e^t$ we get
$$C\int_\Sigma\min\left\{\widehat V_1,\widehat V_2\right\}dv_g\le \int_\Sigma\min\left\{\widehat V_1,\widehat V_2\right\}e^{\frac{2s_1+s_2}3}dv_g\le$$
$$\le\frac{2}3\int_\Sigma\widehat V_1e^{s_1}dv_g+\frac{1}3\int_\Sigma\widehat V_2e^{s_2}dv_g=\frac{2}3\int_\Sigma r_1dv_g+\frac{1}3\int_\Sigma r_2dv_g<+\infty.$$
\\
Therefore, for any $x_0\in S$ there exists $i\in\{1,2\}$ such that $\sum_{j=1}^2a_{ij}\sigma_j(x_0)<4\pi(1+\alpha_i(x_0))$. Fix $x_0$ and suppose, without loss of generality, that this is true for $i=1$. This implies that $r_1\in L^p(B_r(x_0))$ for small $r$, so for $x\in B_{\frac{r}2}(x_0)$ we have
\begin{eqnarray*}
s_2(x)&=&\int_\Sigma G_x(y)2r_2(y)dv_g(y)-\int_{B_r(x_0)} G_x(y)r_1(y)dv_g(y)\\
&-&\int_{\Sigma\backslash B_r(x_0)}G_x(y)r_1(y)dv_g(y)\\
&\ge&-C-\sup_{z\in\Sigma}\|G_z\|_{L^{p'}(\Sigma)}\|r_1\|_{L^p(B_r(x_0))}\\
&-&\sup_{z\in B_\frac{r}2(x_0)}\|G_z\|_{L^\infty(\Sigma\backslash B_r(x_0))}\|r_1\|_{L^1(\Sigma)}\\
&\ge&-C'.
\end{eqnarray*}
Therefore, arguing as before, we must have $\sum_{j=1}^2a_{2j}\sigma_j(x_0)<4\pi(1+\alpha_2(x_0))$ and $r_2\in L^p\left(B_{\frac{r}2}(x_0)\right)$. This implies $-\Delta s_i\in L^p\left(B_{\frac{r}2}(x_0)\right)$ for both $i$'s. Hence, being $x_0$ arbitrary and $-\Delta s_i\in L^p_{loc}(\Sigma\backslash S)$, by elliptic estimates the proof is complete.
\end{proof}

From Lemmas $\ref{green}$ and $\ref{residuo}$ we can deduce, through a Pohozaev identity, the following information about the local blow-up values. This was explicitly done in \cite{jw,lwz}.

\begin{lemma}\label{ellisse}
If $x_0\in S$ then 
$$
\sigma_1^2(x_0) +\sigma_2^2(x_0)-\sigma_1(x_0)\sigma_2(x_0)= 4\pi (1+\alpha_1(x_0))\sigma_1(x_0)+  4\pi (1+\alpha_2(x_0))\sigma_2(x_0).
$$ 
\end{lemma}

\begin{lemma}
If $x_0\in S_1\cap S_2$ then there exists $i$ such that $\sum_{j=1}^2a_{ij}\sigma_j(x_0)\ge 4\pi (1+\alpha_i(x_0))$.
\end{lemma}

\begin{proof}
Suppose the statement is not true. Then, by Lemmas $\ref{residuo}$ and $\ref{ellisse}$, we would have
\begin{equation}\label{conditions}
\left\{\begin{array}{l}2\sigma_1(x_0)-\sigma_2(x_0)<4\pi(1+\alpha_1(x_0))\\2\sigma_2(x_0)-\sigma_1(x_0)<4\pi(1+\alpha_2(x_0))\\\sigma_1^2(x_0) +\sigma_2^2(x_0)-\sigma_1(x_0)\sigma_2(x_0)=\\= 4\pi (1+\alpha_1(x_0))\sigma_1(x_0)+  4\pi (1+\alpha_2(x_0))\sigma_2(x_0)\end{array}\right.,
\end{equation}
which has no solution between positive $\sigma_1(x_0),\sigma_2(x_0)$.\\
In fact, by multiplying the first equation by $\dis{\frac{\sigma_1(x_0)}2}$ and the second by $\dis{\frac{\sigma_2(x_0)}2}$ and summing, we get
$$\sigma_1^2(x_0) +\sigma_2^2(x_0)-\sigma_1(x_0)\sigma_2(x_0)< 2\pi (1+\alpha_1(x_0))\sigma_1(x_0)+  2\pi (1+\alpha_2(x_0))\sigma_2(x_0),$$
which contradicts the third equation.\\
The scenario is described by the picture.

\begin{figure}[h!]
\center
\includegraphics{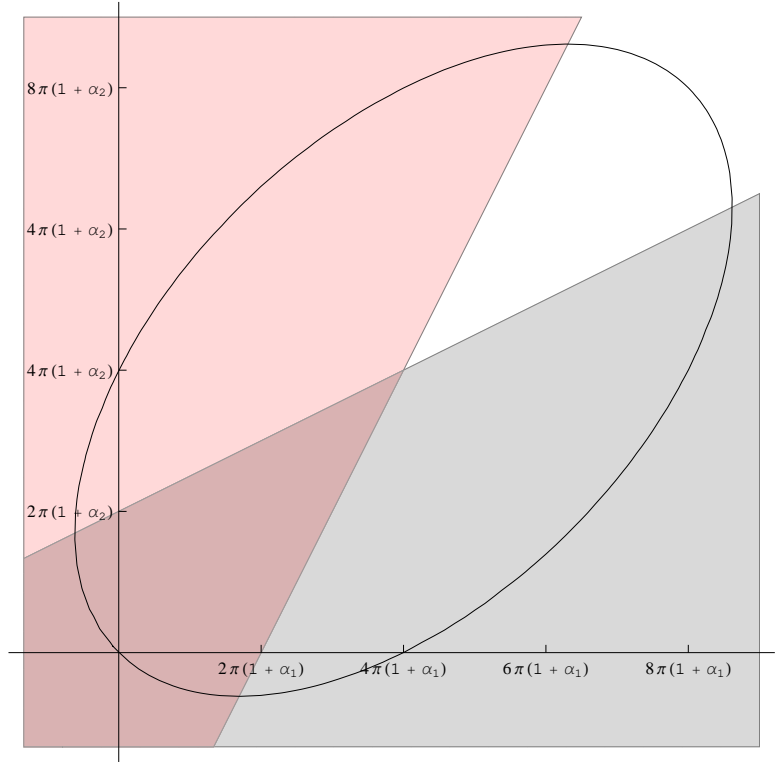}
\caption{The algebraic conditions $\eqref{conditions}$ satisfied by $\sigma_1(x_0),\sigma_2(x_0)$}
\end{figure}

\end{proof}
\begin{cor}
\label{rhoi}
Let $w_n$ be a sequence of solutions of \eqref{eqw}. If $S\neq \0$ then either $r_1\equiv 0$ or $r_2\equiv 0$. In particular there exists $i\in\{1,2\}$ such that $\ov{\rho}_i = \sum_{x\in S_i} \sigma_i(x)$.
\end{cor}

\begin{proof}[Proof of Theorems $\ref{comp}$ and $\ref{compsing}$]$ $\\
Let $u_n$ be a sequence of solutions of $\eqref{toda1}$ with $\rho_i=\rho_{i,n}\underset{n\to+\infty}\ra\ov\rho_i$ and $\int_\Sigma u_{1,n}dv_g=\int_\Sigma u_{2,n}dv_g=0$ and let $w_{i,n}$ be defined by $\eqref{wi}$.\\
If both $w_{1,n}$ and $w_{2,n}$ are bounded from above, then by standard estimates $u_n$ is bounded in $W^{2,p}(\Sigma)$, hence is compact in $H^1(\Sigma)$.\\
Otherwise, from Corollary $\ref{rhoi}$ we must have $\ov{\rho}_i = \sum_{x\in S_i} \sigma_i(x)$ for some $i\in\{1,2\}$. In the regular case, from Theorem $\ref{jlw}$ follows that $\rho_i$ must be an integer multiple of $4\pi$, hence the proof of Theorem $\ref{comp}$ is complete.\\
In the singular case, local blow-up values at regular points are still defined by $\eqref{sigma}$, whereas for any $j=1,\dots,l$ there exists a finite $\Gamma_j$ such that $(\sigma_1(p_j),\sigma_2(p_j))\in\Gamma_j$. Therefore, it must hold
$$\rho_i\in\Lambda_i:=\left\{4\pi k+\sum_{j=1}^ln_j\sigma_j,\;k\in\mathbb N,\;n_j\in\{0,1\},\;\sigma_j\in\Pi_i(\Gamma_j)\right\},$$
where $\Pi_i$ is the projection on the $i^{th}$ component; being $\Lambda_i$ discrete we can also conclude the proof of Theorem $\ref{compsing}$.
\end{proof}

\section*{Acknowledgments}
The authors have been supported by the PRIN project \textit{Variational and perturbative aspects of nonlinear differential problems}.

\bibliographystyle{abbrv}
\bibliography{finale}

\end{document}